\documentclass[review]{elsarticle}

\usepackage{lineno,hyperref}
\usepackage{epsfig}
\usepackage{epstopdf}
\usepackage{graphicx}
\usepackage{subfigure}
\usepackage{float}
\usepackage{multirow}
\usepackage{amsmath,amsthm}
\usepackage{mathrsfs}
\usepackage{amssymb}
\usepackage{color}
\usepackage{makecell}

\include{amsmath}

\newtheorem{lemma}{Lemma}
\newtheorem{theorem}{Theorem}
\usepackage{caption}
\biboptions{numbers,sort&compress}

 \newtheorem{cor}{Corollary}[section]
\numberwithin{equation}{section} %to rank the equation in every section
\begin{document}

\captionsetup[figure]{labelfont={bf},labelformat={default},labelsep=period,name={Fig.}}

\begin{frontmatter}

\title{Error estimates of finite difference methods for the Dirac equation in the massless and nonrelativistic regime}


\author[csrc]{Ying Ma}

\author[nus,lbl]{Jia Yin\corref{corauthor}}
\cortext[corauthor]{Corresponding author}
\ead{jiayin@lbl.gov}

\address[csrc]{Beijing Computational Science Research Center,
Beijing 100193, PR China}

\address[nus]{Department of Mathematics, National University of
Singapore, Singapore 119076, Singapore}

\address[lbl]{Computational Research Division, Lawrence Berkeley National Laboratory,
	Berkeley, CA 94720, USA}

\begin{abstract}
We present four frequently used finite difference methods and establish the error bounds for the discretization of the Dirac equation in the massless and nonrelativistic regime, involving a small dimensionless parameter $0< \varepsilon \ll 1$ inversely proportional to the speed of light. In the massless and nonrelativistic regime, the solution exhibits rapid motion in space and is highly oscillatory in time. Specifically, the wavelength of the propagating waves in time is at $O(\varepsilon)$, while in space it is at $O(1)$ with the wave speed at $O(\varepsilon^{-1}).$ We adopt one leap-frog, two semi-implicit, and one conservative Crank-Nicolson finite difference methods to numerically discretize the Dirac equation in one dimension and establish rigorously the error estimates which depend explicitly on the time step $\tau$, mesh size $h$, as well as the small parameter $\varepsilon$. The error bounds indicate that, to obtain the `correct' numerical solution in the massless and nonrelativistic regime, i.e. $0<\varepsilon\ll1$, all these finite difference methods share the same $\varepsilon$-scalability as time step $\tau=O(\varepsilon^{3/2})$ and mesh size $h=O(\varepsilon^{1/2})$. A large number of numerical results are reported to verify the error estimates.
\end{abstract}

\begin{keyword}
Dirac equation \sep massless and nonrelativistic regime \sep finite difference method \sep oscillatory in time \sep rapid motion in space
%\MSC[2010] 00-01\sep  99-00
\end{keyword}

\end{frontmatter}

\linenumbers

\section{Introduction}
The Dirac equation, which plays a fundamental role in particle physics and mathematics, was proposed by the British physicist Paul Adrien Maurice Dirac in 1928 \cite{D28,D58,T92}. As a relativistic wave equation, the Dirac equation predicted the existence of antimatter which was observed in experiments in 1932 \cite{A33}. Moreover, it is also used to describe the fine details of the hydrogen spectrum, and it has been adopted to describe spin-1/2 massive particles, such as positrons, electrons, muons, neutrons, neutrinos, protons, etc. In addition, the Dirac equation throws light on many scientific phenomena which cannot be explained by classical physics, and provides theoretical support for interpreting some microscopic phenomena and simulating scientific experiments \cite{GKZSBR10}. Since the graphene was first produced in the lab in 2003 \cite{FLB14,NGMJKGDF05}, the Dirac equation has been extensively applied to study the structures and dynamical properties of graphene, graphite, topological insulators and other two dimensional materials. With the progress made in recent experiments, the study of the Dirac equation presents prospective and important scientific applications.
%in 2003: AMPGMKWTNLG11
In this paper, we consider the Dirac equation in the massless and nonrelativistic regime on the torus $\mathbb{T}^d\,(d=1,2,3)$ as following
\begin{equation}\label{tomai}
i \partial_{t} \varPsi =\bigg(-\frac{i}{\varepsilon} \sum_{j=1}^{d} \alpha_{j} \partial_{j}+\frac{1}{\varepsilon} \beta\bigg) \varPsi+\bigg(V(t, \mathbf{x}) I_{4}-\sum_{j=1}^{d} A_{j}(t, \mathbf{x}) \alpha_{j}\bigg) \varPsi, \qquad \textbf{x}\in\mathbb{T}^d,
\end{equation}
where $\textbf{x}=(x_1, ..., x_d)^T\in \mathbb{T}^d $ is the spatial coordinate, $t$ is time, $\partial_{j}=\partial_{x_j}\,(j=1,...,d),$ $i=\sqrt{-1}$, $\varepsilon:=\frac{x_s}{t_s c}\in(0,1]$ is a dimensionless parameter which is inversely proportional to the speed of light $c$. In the expression of $\varepsilon$, $x_s$ and $t_s$ are the dimensionless length and time unit, respectively. $\varPsi:=\varPsi(t, \mathbf{x})=(\psi_1(t, \mathbf{x}),\psi_2(t, \mathbf{x}),\psi_3(t, \mathbf{x}),\psi_4(t, \mathbf{x}))^T$ $\in \mathbb{C}^4$ represents the complex-valued spinor wave function, $V:=V(t, \mathbf{x})$ is the electric potential, while $\mathbf{A}:=\mathbf{A}(t, \mathbf{x})=(A_1(t, \mathbf{x}), ..., A_d(t, \mathbf{x}))^T$ is the magnetic potential. The electromagnetic potentials are given real-valued functions. Besides, $I_n\,(n\in \mathbb{N})$ is the $n\times n$ identity matrix, and the Dirac matrices $\alpha_j\,(j=1,2,3)$, $\beta$ are all $4\times 4$ matrices which are defined as
\begin{equation}\label{mat}
\alpha_{j}=\left(\begin{array}{cc}
{\mathbf{0}} & {\sigma_{j}} \\
{\sigma_{j}} & {\mathbf{0}}
\end{array}\right), \quad \beta=\left(\begin{array}{cc}
{I_{2}} & {\mathbf{0}} \\
{\mathbf{0}} & {-I_{2}}
\end{array}\right),\qquad j=1,2,3,
\end{equation}
where the Pauli matrices $\sigma_j\,(j=1,2,3)$ are defined as follows
\begin{equation}\label{pau}
\sigma_{1}=\Bigg(\begin{array}{cc}
{0} & {1} \\
{1} & {0}
\end{array}\Bigg), \quad
\sigma_{2}=\Bigg(\begin{array}{cc}
{0} & {-i} \\
{i} & {0}
\end{array}\Bigg), \quad
\sigma_3=\Bigg(\begin{array}{cc}
{1} & {0} \\
{0} & {-1}
\end{array}\Bigg).
\end{equation}

As stated in \cite{BCJT17}, in the case of one dimension (1D) and two dimensions (2D) ($d=1,2$), the Dirac equation \eqref{tomai} can be simplified as
\begin{equation}\label{maiequ}
i \partial_{t} \varPhi=\bigg(-\frac{i}{\varepsilon} \sum_{j=1}^{d} \sigma_{j} \partial_{j}+\frac{1}{\varepsilon} \sigma_{3}\bigg) \varPhi+\bigg(V(t, \mathbf{x}) I_{2}-\sum_{j=1}^{d} A_{j}(t, \mathbf{x}) \sigma_{j}\bigg) \varPhi,\qquad \textbf{x}\in \mathbb{T}^d,
\end{equation}
where $\varPhi:=\varPhi(t, \textbf{x})=(\phi_1(t, \textbf{x}),\phi_2(t, \textbf{x}))^T\in\mathbb{C}^2$. To study its dynamics behavior, the initial condition is usually taken as
\begin{equation}\label{2dint}
\varPhi(t=0, \textbf{x})=\varPhi_0(\textbf{x}),\quad \textbf{x}\in \mathbb{T}^d.
\end{equation}
The Dirac equation \eqref{maiequ} maintains the total mass conservation, i.e.
\begin{equation}\label{con}
\begin{aligned}
\|\varPhi(t, \cdot)\|^{2} &:=\int_{\mathbb{T}^{d}}|\varPhi(t, \mathbf{x})|^{2} d \mathbf{x}=\int_{\mathbb{T}^{d}} \sum_{j=1}^2\left|\varPhi_{j}(t, \mathbf{x})\right|^{2} d \mathbf{x} \\& \equiv\|\varPhi(0, \cdot)\|^{2}=\left\|\varPhi_{0}\right\|^{2}, \quad t \geq 0.
\end{aligned}
\end{equation}

Introduce the total density
\begin{equation}\label{pas}
\rho(t,\textbf{x})=\sum_{l=1}^{2}\rho_l(t,\textbf{x})=\varPhi(t,\textbf{x})^*
\varPhi(t,\textbf{x}),\quad \textbf{x} \in\mathbb{T}^d,
\end{equation}
where $\varPhi^*=\overline{\varPhi}^T$ with $\overline{\varPhi}$ benig the complex conjugate of $\varPhi$, and the $l$-th component position density $\rho_l(t,\textbf{x})=|\phi_l(t,\textbf{x})|^2$ for $l=1,2$.
Besides, we define the current density $\textbf{J}(t,\textbf{x})=(J_1(t,\textbf{x}),\cdot\cdot\cdot,J_d(t,\textbf{x}))^T$ in the following
\begin{equation}\label{cde}
J_j(t,\textbf{x})=\frac{1}{\varepsilon}\varPhi(t,\textbf{x})^*
\sigma_j\varPhi(t,\textbf{x}),\quad j=1,\ldots,d.
\end{equation}
Then from the Dirac equation \eqref{maiequ}, we can derive the conservation law as below
\begin{equation}\label{col}
\partial_t\rho(t,\textbf{x})+\nabla\cdot\textbf{J}(t,\textbf{x})=0,\quad
\textbf{x}\in\mathbb{T}^d,\quad t\ge0.
\end{equation}
Here we notice when the electric potential $V$ is perturbed by a real constant $V^0$, i.e., $V \rightarrow V+V^0,$ the wave function can be expressed as $\varPhi(t,\textbf{x}) \rightarrow e^{-iV^0t}\varPhi(t,\textbf{x})$,  implying that the total density $\rho$ and the position density of each component $\rho_l$, $(l=1,2)$ are all unchanged. Furthermore, when $d=1$ and the magnetic potential $A_1$ is perturbed by a real constant $A_1^0$, i.e., $A_1 \longrightarrow A_1+A_1^0,$ then the solution can be expressed as $\varPhi(t,\textbf{x}) \rightarrow e^{iA_1^0t\sigma_1}\varPhi(t,\textbf{x}),$ which implies that the total density $\rho$ is unchanged. However, this property is not valid for $d=2.$ If the electromagnetic potentials are time-independent, i.e., $V(t,\textbf{x})=V(\textbf{x})$ and $A_j(t,\textbf{x})=A_j(\textbf{x}),\;j=1,2,$
then we can obtain that the energy functional remains conserved as
\begin{equation}\label{eco}
\begin{aligned}
E(\varPhi(t,\cdot)) :&=\int_{\mathbb{T}^{d}}\Big(-\frac{i}{\varepsilon} \sum_{j=1}^{d} \varPhi^{*} \sigma_{j} \partial_{j} \varPhi+\frac{1}{\varepsilon} \varPhi^{*} \sigma_{3} \varPhi+V(\mathbf{x})|\varPhi|^{2}-\sum_{j=1}^{d} A_{j}(\mathbf{x}) \varPhi^{*} \sigma_{j} \varPhi\Big) d \mathbf{x} \\
& \equiv E(\varPhi_0), \quad t \geq 0.
\end{aligned}
\end{equation}

When $\varepsilon=1$ in \eqref{maiequ} (or \eqref{tomai}), it collapses to the standard Dirac equation. A large quantity of analytical and numerical results have been devoted in this regime in literatures. For details, we refer to \cite{RK63,GMMP00,AT17,A92,BHM14,G15,HPAS14,AL17,BSG99,LLS17} and references therein. We remark here that there have been massive numerical results for the linear/nonlinear Dirac equations in different parameter regimes, such as nonrelativistic regime \cite{FW50,BCJY16,BCJT16,BCJT17,CW18,BCY20,BY19,CW19}, semiclassical regime \cite{WHJY12,BK99,MY19}, etc.
%refer to: HPAA14
%nonrelativistic: GNP89,LMZ17,N92,BM98,W06
%semiclassical: BSM02,K04,S00,JMS11

When $0<\varepsilon\ll1$ in \eqref{maiequ} (or \eqref{tomai}), in the
massless and nonrelativistic regime, the Dirac equation \eqref{maiequ} (or \eqref{tomai}) is a highly oscillatory dispersive partial differential equation \cite{BY19}. It propagates waves with wavelength $O(\varepsilon)$ in time and $O(1)$ in space, while the wave speed in space is at $O(1/\varepsilon)$. In other words, the waves are highly oscillatory in time and are rapidly propagating in space. To illustrate this, Fig. \ref{osctx} plots the wave function solution of \eqref{maiequ} with $d=1,\,V(t,x)=\frac{1}{2+{\textrm{sin}(\pi x)}},\,A_{1}(t, x)=\frac{1}{1+{\textrm{cos}^{2}(\pi x)}}$ and initial data $\varPhi_0(x)=\big(\textrm{sin}(\pi (x+1)),\textrm{cos}(\pi (x+1))\big)^T$ for various $\varepsilon.$

\begin{figure}[htp]
\centering
\subfigure{\includegraphics[width=0.495\textwidth]{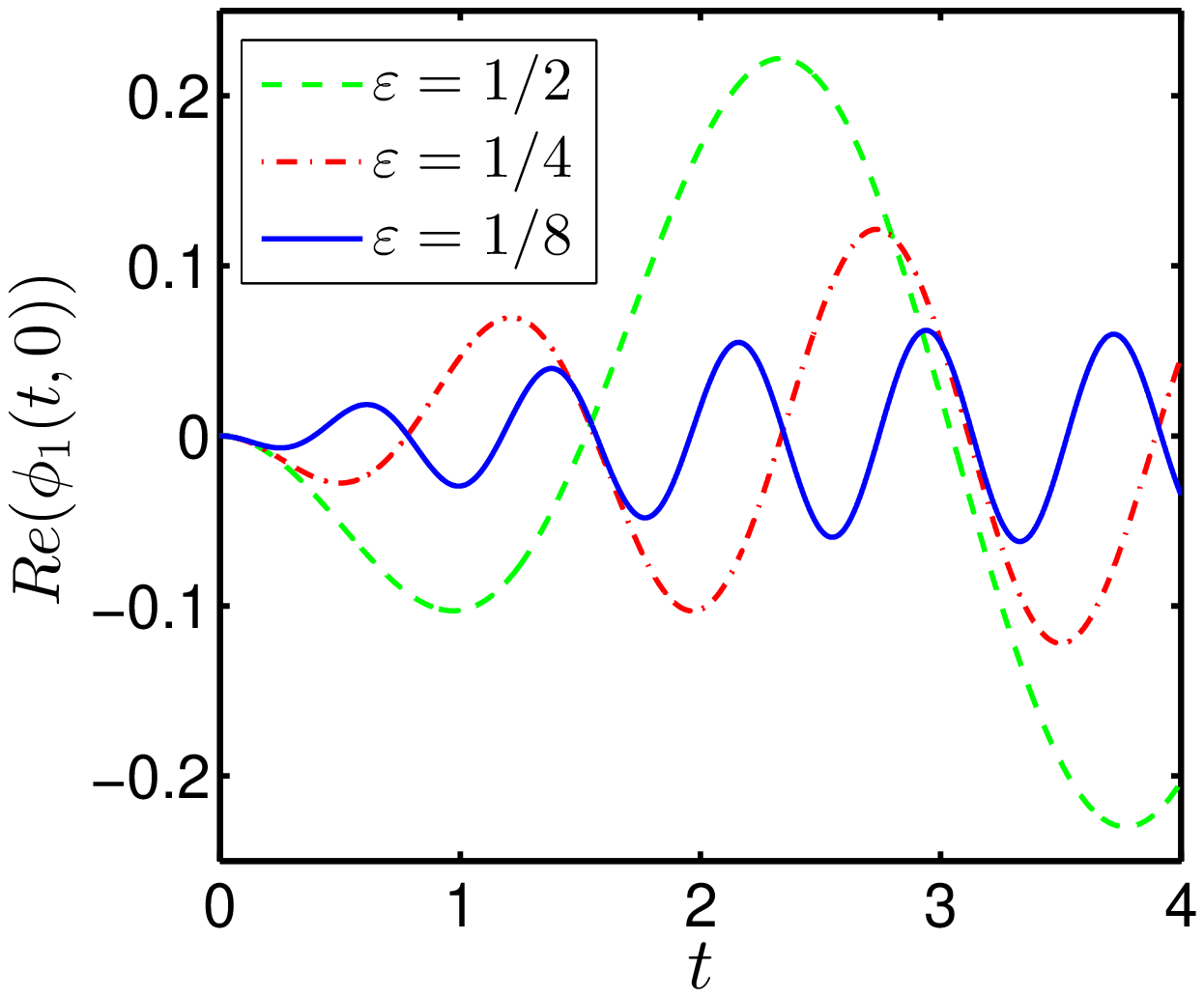}}
\subfigure{\includegraphics[width=0.495\textwidth]{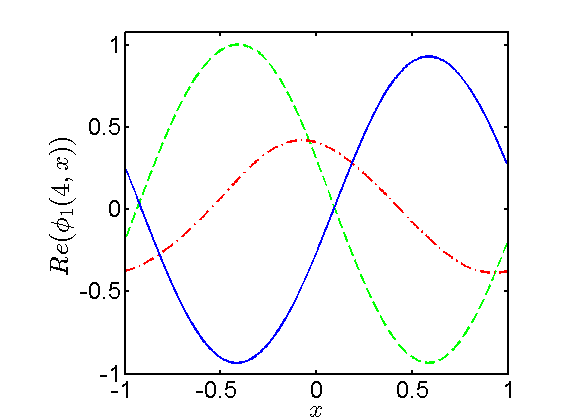}}
\captionsetup{width=\linewidth}
\caption{The real part of the wave function $\phi_1(t,x=0)$ and $\phi_1(t=4,x)$ for the Dirac equation \eqref{maiequ} in 1D with various $\varepsilon$. }
\label{osctx}
\end{figure}

For the Dirac equation in certain parameter regimes, the highly oscillatory nature of the solution causes serious numerical burdens, which makes the numerical approximation for the Dirac equation \eqref{maiequ} (or \eqref{tomai}) costly and extremely challenging. As a result, it is very important to design effective numerical methods. To our best knowledge, there are few numerical methods and research achievements for the Dirac equation \eqref{maiequ} (or \eqref{tomai}) in the massless and nonrelativistic regime. In this paper, the main purpose is to investigate the efficiency and to prove the error bounds of the finite difference methods for the Dirac equation in the massless and nonrelativistic regime. We analyze the stability and convergence of four fully explicit/semi-implicit/implicit finite difference methods. Specifically, we focus on how the error estimates are explicitly dependent on the time step $\tau$, the mesh size $h$, as well as the small parameter $\varepsilon$. Based on our error estimates, if we want to obtain the `correct' numerical solutions in the massless and nonrelativistic regime $(0<\varepsilon\le1)$, the meshing strategies (or $\varepsilon$-scalability) for the above four finite difference methods should all be $\tau=O(\varepsilon^{3/2})$ and $h=O(\varepsilon^{1/2}).$  The performance of various methods is reported by numerical results.

The rest of this paper is arranged as follows. In Section 2, we present the Crank-Nicolson finite difference (CNFD) method for the Dirac equation in the massless and nonrelativistic regime, show its mass and energy conservation, and establish its error bounds. Moreover, extensive numerical results are reported to confirm the error estimates and to demonstrate that our error bounds are sharp. In Section 3, we propose a semi-implicit finite difference (SIFD1) method for the problem, find its stability condition, prove its error bounds and report its numerical results. Similar results for another semi-implicit finite difference (SIFD2) method and the leap-frog finite difference (LFFD) method are presented in Section 4. Finally, some conclusions are drawn in Section 5.

In order to simplify the notations, we adopt the standard Sobolev spaces and norms, and the notation $p \lesssim q$ represents that there exists a generic positive constant $C>0$ independent of $\varepsilon,\,\tau,\,h,$ such that $|p|\le Cq.$ In the following discussion, we will take the 1D Dirac equation (\eqref{maiequ} with $d=1$) as an example to present the related stabilities and convergence analysis of the finite difference methods. The results can be generalized to the 2D case of \eqref{maiequ} and the cases $d=1, 2, 3$ of the four-component Dirac equation \eqref{tomai} directly, and the conclusions remain valid without modifications.

In the following, we consider the 1D Dirac equation \eqref{maiequ} on a bounded domain with periodic boundary conditions
\begin{eqnarray}\label{x1}
i{\partial _t}\varPhi =\left(-\frac{i}{\varepsilon} \sigma_{1} \partial_{x}+\frac{1}{\varepsilon} \sigma_{3}\right) \varPhi+\Big(V(t, x) I_{2}-A_{1}(t, x) \sigma_{1}\Big) \varPhi, \quad t>0,\,\, x \in\Omega,  \\ \label{x2}
\varPhi(t, a) = \varPhi(t, b),\,\; \partial_{x} \varPhi(t, a)=\partial_{x} \varPhi(t, b), \,t \geq 0 ;\quad  \varPhi(0, x)=\varPhi_{0}(x),\,\, x \in \bar{\Omega},
\end{eqnarray}
where $\Omega=(a,b),\,\varPhi:=\varPhi(t, x),\, \varPhi_{0}(a)=\varPhi_{0}(b) \text { and } \varPhi_{0}^{\prime}(a)=\varPhi_{0}^{\prime}(b).$
\section{A Crank-Nicolson finite difference (CNFD) method and its error estimate}
In this section, for the Dirac equation \eqref{x1}-\eqref{x2}, we adopt the Crank-Nicolson finite difference (CNFD) method.

\subsection{The CNFD method }
We choose the time step $\tau:=\Delta t>0$ and the mesh size $h:=\Delta x=\frac{b-a}{M}$, where $M$ is a positive integer, and define the uniform time steps and grid points as following:
\begin{equation*}
t_n:=n\tau,\quad n=0,1,2,\ldots \quad x_j:=a+jh,\quad j=0,1,\ldots,M.
\end{equation*}
Denote $\varPhi_j^n$ as the numerical approximation of $\varPhi(t_n,x_j),$ $V_j^n=V(t_n,x_j),\,V_j^{n+\frac{1}{2}}=V(t_n+\tau/2,x_j),\,
A_{1,j}^{n}=A_1(t_{n},x_j)$ and $A_{1,j}^{n+\frac{1}{2}}=A_1(t_n+\tau/2,x_j)$ for $n\ge 0$ and $0\le j\le M.$
Denote $\varPhi^n=(\varPhi_0^n,\varPhi_1^n,\cdot\cdot\cdot,\varPhi_M^n)^T\in X_M$, where $X_M=\{U=(U_0,U_1,\ldots,U_M)^T|U_j\in\mathbb{C}^2,
	\,j=0,1,\ldots,M,\,U_0=U_M\}$, as the solution vector at $t=t_n$. Let us introduce the discretization operators of the finite difference method for $n\ge 0$ and $j=0,1,\cdot\cdot\cdot,M$ as follows:
\begin{equation*}\label{yx}
\delta_{t}^{+} \varPhi_{j}^{n}=\frac{\varPhi_{j}^{n+1}-\varPhi_{j}^{n}}{\tau}, \quad \delta_{t} \varPhi_{j}^{n}=\frac{\varPhi_{j}^{n+1}-\varPhi_{j}^{n-1}}{2 \tau}, \quad \delta_{x} \varPhi_{j}^{n}=\frac{\varPhi_{j+1}^{n}-\varPhi_{j-1}^{n}}{2 h},
\end{equation*}
and
\begin{equation*}\label{z}
\varPhi_{j}^{n+\frac{1}{2}}=\frac{\varPhi_{j}^{n+1}+\varPhi_{j}^{n}}{2}.
\end{equation*}

In order to discretize the Dirac equation \eqref{x1} for $n\ge 0,\,j=0,1,\cdot\cdot\cdot,M-1,$ we consider the following frequently used CNFD scheme
\begin{equation}\label{fuz4}
i \delta_{t}^{+} \varPhi_{j}^{n}=\frac{1}{\varepsilon}\Big(-i \sigma_{1} \delta_{x} +\sigma_{3} \Big) \varPhi_{j}^{n+\frac{1}{2}}+\left(V_{j}^{n+\frac{1}{2}} I_{2}-A_{1, j}^{n+\frac{1}{2}} \sigma_{1}\right) \varPhi_{j}^{n+\frac{1}{2}}.
\end{equation}
The boundary and initial conditions of \eqref{x2} are discretized as below:
\begin{equation}\label{zbcn}
\varPhi_{M}^{n+1}=\varPhi_{0}^{n+1},~\, \varPhi_{-1}^{n+1}=\varPhi_{M-1}^{n+1},~\, n \geq 0 ; \quad \varPhi_{j}^{0}=\varPhi_{0}\left(x_{j}\right),~\, j=0,1, \ldots, M.
\end{equation}
Here we notice that the CNFD method is time symmetric, which means that it is unchanged under $n+1\leftrightarrow n$ and $\tau\leftrightarrow-\tau$. The CNFD method is unconditionally stable, in other words, it is stable for any $\tau$, $h>0$ and $0<\varepsilon\le1.$ The memory cost of the CNFD method \eqref{fuz4} is $O(M)$. It is implicit and at each time step for $n\ge0$, its corresponding linear system is coupled in order that it needs to be solved by means of either an iterative solver or a direct solver. Hence, the computational cost per step mainly
depends on its linear system solver, which is generally much larger than $O(M)$, especially in 2D and 3D.

\subsection{Mass and energy conservation}
If $U\in X_M$, then we take $U_{-1}=U_{M-1}$ and $U_{M+1}=U_1$ if they are involved. In $X_M$, define the standard $l^2$ and $l^\infty$ norms as below
\begin{equation}\label{yy}
\|U\|_{l^{2}}^{2}=h \sum_{j=0}^{M-1}\left|U_{j}\right|^{2}, \quad \|U\|_{l^{\infty}}^{2}=\mathop {\max }\limits_{0 \le j \le M } \left| {{U_j}} \right|, \quad U \in X_{M},
\end{equation}
For the CNFD method \eqref{fuz4} to \eqref{x1}-\eqref{x2}, we obtain the mass and energy conservative properties as below.
\begin{lemma}\label{lem}
The CNFD method \eqref{fuz4} conserves the mass in the discretized level, that is
\begin{equation}\label{LME1}
\left\|\varPhi^{n}\right\|_{l^{2}}^{2}:=h \sum_{j=0}^{M-1}\left|\varPhi_{j}^{n}\right|^{2} \equiv h \sum_{j=0}^{M-1}\left|\varPhi_{j}^{0}\right|^{2}=
\left\|\varPhi^{0}\right\|_{l^{2}}^{2}=
h\sum_{j=0}^{M-1}\left|\varPhi_{0}\left(x_{j}\right)\right|^{2},
\quad n \geq 0.
\end{equation}
Moreover, if both $\,V(t,x)=V(x)$ and $A_1(t,x)=A_1(x)$ remain time independent, the method \eqref{fuz4} conserves the energy as well,
\begin{equation}\label{LME2}
\begin{aligned} E_{h}^{n} &=h \sum_{j=0}^{M-1}\left[-\frac{i}{\varepsilon}\left(\varPhi_{j}^{n}\right)^{*} \sigma_{1} \delta_{x} \varPhi_{j}^{n}+\frac{1}{\varepsilon}\left(\varPhi_{j}^{n}\right)^{*} \sigma_{3} \varPhi_{j}^{n}+V_{j}\left|\varPhi_{j}^{n}\right|^{2}-A_{1, j}\left(\varPhi_{j}^{n}\right)^{*} \sigma_{1} \varPhi_{j}^{n}\right]
\\ & \equiv E_{h}^{0}, \quad n \geq 0,
\end{aligned}
\end{equation}
in which $V_j=V(x_j)$ and $A_{1,j}=A_1(x_j)$ for $j=0,1,\cdots,M.$
\end{lemma}
\begin{proof}
(\romannumeral1) First of all, we can prove the mass conservation in \eqref{LME1}. Multiply both sides of equation \eqref{fuz4} from the left by $h\tau(\varPhi_j^{n+\frac{1}{2}})^*$ and take its imaginary part, we obtain that
\begin{equation}\label{MEP1}
h\left|\varPhi_{j}^{n+1}\right|^{2}=h\left|\varPhi_{j}^{n}\right|^{2}-\frac{\tau h}{2 \varepsilon}\left[\left(\varPhi_{j}^{n+\frac{1}{2}}\right)^{*} \sigma_{1} \delta_{x} \varPhi_{j}^{n+\frac{1}{2}}+\left(\varPhi_{j}^{n+\frac{1}{2}}\right)^{T} \sigma_{1} \delta_{x} \overline{\varPhi}_{j}^{n+\frac{1}{2}}\right],\quad n\ge0,\; j=0,1,\cdot\cdot\cdot M-1.
\end{equation}
Summing up \eqref{MEP1} for $j=0,1,\cdot\cdot\cdot,M-1,$ as well as noticing \eqref{pau}, we have
\begin{equation}\label{MEP2}
\begin{aligned}\left\|\varPhi^{n+1}\right\|_{l^{2}}^{2}=
&\left\|\varPhi^{n}\right\|_{l^{2}}^{2}-\frac{\tau h}{2 \varepsilon}  \sum_{j=0}^{M-1}\left[\left(\varPhi_{j}^{n+\frac{1}{2}}\right)^{*} \sigma_{1} \delta_{x} \varPhi_{j}^{n+\frac{1}{2}}+\left(\varPhi_{j}^{n+\frac{1}{2}}\right)^{T} \sigma_{1} \delta_{x}\overline{\varPhi}_{j}^{n+\frac{1}{2}}\right] \\
=&\left\|\varPhi^{n}\right\|_{l^{2}}^{2}-\frac{\tau}{4 \varepsilon} \sum_{j=0}^{M-1}\left[\left(\varPhi_{j}^{n+\frac{1}{2}}\right)^{*} \sigma_{1} \varPhi_{j+1}^{n+\frac{1}{2}}+\left(\varPhi_{j}^{n+\frac{1}{2}}\right)^{T} \sigma_{1} \overline{\varPhi}_{j+1}^{n+\frac{1}{2}}\right.\\
&\qquad\quad\ -\left.\left(\varPhi_{j+1}^{n+\frac{1}{2}}\right)^{*} \sigma_{1} \varPhi_{j}^{n+\frac{1}{2}}-\left(\varPhi_{j+1}^{n+\frac{1}{2}}\right)^{T} \sigma_{1} \overline{\varPhi}_{j}^{n+\frac{1}{2}}\right] \\=&\left\|\varPhi^{n}\right\|_{l^{2}}^{2}, \quad n \geq 0, \end{aligned}
\end{equation}
which directly gives \eqref{LME1} by induction.
\\
(\romannumeral2) Secondly, we prove the energy conservation in \eqref{LME2}. Multiply both sides of \eqref{fuz4} from the left by $2h(\varPhi_j^{n+1}-\varPhi_j^n)^*$ and take its real part, we obtain for $j=0,1,\cdot\cdot\cdot M-1,$
\begin{equation}\label{EP1}
\begin{array}{l}
{-h\operatorname{Re}\left[\dfrac{i}{\varepsilon}\left(\varPhi_{j}^{n+1}-\varPhi_{j}^{n}\right)^{*} \sigma_{1} \delta_{x}\left(\varPhi_{j}^{n+1}+\varPhi_{j}^{n}\right)\right]
+\dfrac{h}{\varepsilon}\left[\left(\varPhi_{j}^{n+1}\right)^{*} \sigma_{3} \varPhi_{j}^{n+1}-\left(\varPhi_{j}^{n}\right)^{*} \sigma_{3} \varPhi_{j}^{n}\right]} \\
{+hV_{j}\left(\left|\varPhi_{j}^{n+1}\right|^{2}-\left|\varPhi_{j}^{n}\right|^{2}\right)-h A_{1, j}\left[\left(\varPhi_{j}^{n+1}\right)^{*} \sigma_{1} \varPhi_{j}^{n+1}-\left(\varPhi_{j}^{n}\right)^{*} \sigma_{1} \varPhi_{j}^{n}\right]=0},\quad n\ge0.
\end{array}
\end{equation}
Then sum up \eqref{EP1} for $j=0,1,\cdot\cdot\cdot,M-1,$  notice the above mass conservation property and the summation by parts formula, we get
\begin{equation}\label{EP2}
\begin{aligned}
& h \sum_{j=0}^{M-1} \operatorname{Re}\left[\frac{i}{\varepsilon}\left(\varPhi_{j}^{n+1}-
\varPhi_{j}^{n}\right)^{*} \sigma_{1} \delta_{x}\left(\varPhi_{j}^{n+1}+\varPhi_{j}^{n}\right)\right] \\=& \operatorname{Re}\bigg[\frac{i h}{\varepsilon} \sum_{j=0}^{M-1}\left(\varPhi_{j}^{n+1}\right)^{*} \sigma_{1} \delta_{x} \varPhi_{j}^{n+1}-\frac{i h}{\varepsilon} \sum_{j=0}^{M-1}\left(\varPhi_{j}^{n}\right)^{*} \sigma_{1} \delta_{x} \varPhi_{j}^{n} \bigg],
\end{aligned}
\end{equation}
and
\begin{equation}\label{EP3}
\begin{aligned}
&-\operatorname{Re}\bigg(\frac{i h}{\varepsilon} \sum_{j=0}^{M-1}\left(\varPhi_{j}^{n+1}\right)^{*} \sigma_{1} \delta_{x} \varPhi_{j}^{n+1}\bigg)
+\frac{h}{\varepsilon}\sum_{j=0}^{M-1}\left(\varPhi_{j}^{n+1}\right)^{*} \sigma_{3}\varPhi_{j}^{n+1}
+h \sum_{j=0}^{M-1} V_{j}\left|\varPhi_{j}^{n+1}\right|^{2}\\
&-h \sum_{j=0}^{M-1} A_{1, j}\left(\varPhi_{j}^{n+1}\right)^{*} \sigma_{1} \varPhi_{j}^{n+1}
=-\operatorname{Re}\bigg(\frac{i h}{\varepsilon} \sum_{j=0}^{M-1}\left(\varPhi_{j}^{n}\right)^{*} \sigma_{1} \delta_{x} \varPhi_{j}^{n}\bigg)+\frac{h}{\varepsilon} \sum_{j=0}^{M-1}\left(\varPhi_{j}^{n}\right)^{*} \sigma_{3} \varPhi_{j}^{n}\\
&+h \sum_{j=0}^{M-1} V_{j}\left|\varPhi_{j}^{n}\right|^{2}
 -h \sum_{j=0}^{M-1} A_{1, j}\left(\varPhi_{j}^{n}\right)^{*} \sigma_{1} \varPhi_{j}^{n},\quad n\ge0,
\end{aligned}
\end{equation}
which directly demonstrates \eqref{LME2}.
\end{proof}
\subsection{Error estimate}
Denote $0<T<T^*$ with $T^*$ being the maximal existence time of the solution, and  $\Omega_T=[0,T]\times\Omega$ with $\Omega=(a,b).$ In order to get the appropriate error estimates, we assume that the exact solution of \eqref{x1} satisfies
$
\varPhi \in C^{3}\big([0, T] ;\left(L^{\infty}(\Omega)\right)^{2}\big) \cap C^{2}\big([0, T] ;\left(W_{p}^{1, \infty}(\Omega)\right)^{2}\big) \cap C^{1}\big([0, T] ;\left(W_{p}^{2, \infty}(\Omega)\right)^{2}\big) \cap C\big([0, T] ;\left(W_{p}^{3, \infty}(\Omega)\right)^{2}\big)$ and
\begin{equation}\label{asu}
(A)\qquad \left\|\frac{\partial^{r+s}}{\partial t^{r} \partial x^{s}} \varPhi\right\|_{L^{\infty}\left([0, T] ;\left(L^{\infty}(\Omega)\right)^{2}\right)} \lesssim \frac{1}{\varepsilon^r},\qquad 0 \leq r \leq 3,~0 \leq r+s \leq 3,
~0<\varepsilon \leq 1,
\end{equation}
in which $W_p^{m,\infty}(\Omega)=\{u|u\in W^{m,\infty}(\Omega),\;\partial_x^l u(a)=\partial_x^l u(b),\;\, l=0,\ldots,m-1\}$ for $m\ge1$ and here the boundary values are understood in the trace sense. In the follow-up discussion,
we will omit $\Omega$ when referring to the space norm taken on $\Omega$. Besides, we assume that the electric and magnetic potentials satisfy $V\in C(\overline{\Omega}_T),\,A_1\in C(\overline{\Omega}_T)$ and we denote
\begin{equation}\label{lefs2}
(B)\qquad V_{\max }:=\max _{(t, x) \in \overline{\Omega}_T}|V(t, x)|, \quad A_{1, \max }:=\max _{(t, x) \in \overline{\Omega}_T}\left|A_{1}(t, x)\right|.
\end{equation}
Here we define the grid error function $\textbf{e}^n=(\textbf{e}_0^n,\textbf{e}_1^n,...,\textbf{e}_M^n)^T\in X_M$ as following:
\begin{equation}\label{errfun}
\textbf{e}_j^n:=\varPhi(t_n,x_j)-\varPhi_j^n,\qquad j=0,1,\cdots,M,~n\ge0,
\end{equation}
in which $\varPhi_j^n$ being the numerical approximation of $\varPhi(t_n,x_j)$ from the finite difference methods. For the CNFD method \eqref{fuz4}, we could derive the error estimates as follow.
\begin{theorem}\label{4th4}
Under the assumptions in (A) and (B), there exist the constants $h_0>0,\,\tau_0>0$ independent of $\varepsilon$ and sufficiently small, such that for any $0<\varepsilon\le1,\,0<h \le h_0$ and $0<\tau\le \tau_0,$ for the CNFD method \eqref{fuz4} with \eqref{zbcn}, we obtain the error estimate on the wave function as below
\begin{equation}\label{tccnfd}
\left\|\mathrm{\textbf{e}}^{n}\right\|_{l^{2}} \lesssim \frac{h^{2}}{\varepsilon}+\frac{\tau^{2}}{\varepsilon^{3}}, \qquad 0 \leq n \leq \frac{T}{\tau}.
\end{equation}
\end{theorem}
\begin{proof}
The local truncation error ${\xi}^n=({\xi}_0^n,{\xi}_1^n,\ldots,
{\xi}_M^n)^T\in X_M$ of the CNFD \eqref{fuz4} with \eqref{zbcn} for $0\le j\le M-1$ and $n\ge0$ is defined as follows
\begin{equation}\label{ccff1}
{\xi}_{j}^{n}:=i \delta_{t}^+ \varPhi(t_n,x_j)+\left[\frac{i}{\varepsilon} \sigma_{1} \delta_{x}
-\Big(\frac{\sigma_{3}}{\varepsilon} +V_{j}^{n+\frac{1}{2}} I_{2}-A_{1, j}^{n+\frac{1}{2}} \sigma_{1}\Big) \right]\varPhi(t_{n+\frac{1}{2}},x_j),
\end{equation}
by using the Taylor expansion and triangle inequality, and by noticing the assumptions (A) and (B), we obtain that
\begin{equation}\label{ccff2}
 \begin{aligned}
|{\xi}_{j}^{n}|
&\le\frac{ \tau^2}{6}\| \partial_{t t t} \varPhi \|_{l^\infty}
+\frac{ \tau^2}{4\varepsilon}\| \partial_{x t t} {\varPhi\|_{l^\infty}}
+\frac{h^{2}}{6\varepsilon} \| \partial_{x x x} {\varPhi\|_{l^\infty}}
+\frac{\tau^{2}}{4}\Big(\frac{1}{\varepsilon}+V_{max}+A_{max} \Big) \|\partial_{t t} \varPhi \|_{l^\infty}\\
&\lesssim \frac{\tau^2}{\varepsilon^3}+ \frac{h^2}{\varepsilon}+ \frac{\tau^2}{\varepsilon^2}\\
&\lesssim \frac{h^2}{\varepsilon}+ \frac{\tau^2}{\varepsilon^3},\qquad j=0,1,\cdot\cdot\cdot,M-1,~ n\ge 0,
\end{aligned}
\end{equation}
hence, we have
\begin{equation}\label{ccff4}
\left\|{\xi}^{n}\right\|_{l^{\infty}}=\max _{0 \leq j \leq M-1}\left|{\xi}_{j}^{n}\right| \lesssim \frac{h^2}{\varepsilon}+\frac{\tau^2}{\varepsilon^3},\qquad
\left\|{\xi}^{n}\right\|_{l^{2}} \lesssim\left\|{\xi}^{n}\right\|_{l^{\infty}} \lesssim \frac{h^{2}}{\varepsilon}+\frac{\tau^2}{\varepsilon^3}, \quad n \geq 0,
\end{equation}
Subtracting \eqref{fuz4} from \eqref{ccff1} and noticing \eqref{errfun}, we obtain the error function with $0\le j\le M-1$ and $n\ge0$ as below
\begin{equation}\label{ccff41}
 i \delta_{t}^+ \mathbf{e}_{j}^{n}=-\frac{i}{\varepsilon} \sigma_{1} \delta_{x} \mathbf{e}_{j}^{n+\frac{1}{2}}+
\left(\frac{\sigma_{3}}{\varepsilon}+V_{j}^{n+\frac{1}{2}} I_{2}-A_{1, j}^{n+\frac{1}{2}} \sigma_{1}\right) \mathbf{e}_{j}^{n+\frac{1}{2}}+{\xi}_{j}^{n},
\end{equation}
here take its initial and boundary conditions as
\begin{equation}\label{cf5}
 \mathbf{e}_{0}^{n}=\mathbf{e}_{M}^{n}, \quad \mathbf{e}_{-1}^{n}=\mathbf{e}_{M-1}^{n}, \quad n \geq 0, \quad \mathbf{e}_{j}^{0}=\mathbf{0}, \quad j=0,1, \ldots, M.
\end{equation}
Multiply $h\tau(\mathbf{e}_j^{n+1}+\mathbf{e}_j^{n})^*$ from the left on both sides of \eqref{ccff41} and take the imaginary part, then sum up for $j=0,1,\cdot\cdot\cdot,M-1$ and use Cauchy inequality again, we obtain
\begin{equation}\label{ccff8}
 \begin{aligned}
\left\|\mathbf{e}^{n+1}\right\|_{l^{2}}^{2}-\left\|\mathbf{e}^{n}\right\|_{l^{2}}^{2}
 & =\tau \operatorname{Im}\bigg[h \sum_{j=0}^{M-1}\left(\mathbf{e}_{j}^{n+1}+\mathbf{e}_{j}^{n}\right)^{*} {{\xi}}_{j}^{n}\bigg]\\
 & \lesssim \tau\left(\|\mathbf{e}^{n+1}\|_{l^2}^2+\|\mathbf{e}^{n}\|_{l^2}^2\right)
 +\tau\|{{\xi}}^{n}\|_{l^2}^2, \quad n \geq 0,
 \end{aligned}
\end{equation}
by noticing \eqref{ccff4} and summing the inequality \eqref{ccff8} for $n=0,1,2,\cdot\cdot\cdot,m-1,$ we obtain that
\begin{equation}\label{ccff9}
\left\|\mathbf{e}^{m}\right\|_{l^{2}}^{2}-\left\|\mathbf{e}^{0}\right\|_{l^{2}}^{2}
\lesssim\tau \sum_{s=0}^{m}\|\mathbf{e}^{s}\|_{l^2}^2
+m\tau\left(\frac{h^{2}}{\varepsilon}+\frac{\tau^2}{\varepsilon^3}\right)^2,
\quad 1 \leq m \leq \frac{T}{\tau},
\end{equation}
where $\left\|\mathbf{e}^{0}\right\|_{l^{2}}^{2}=0.$ By taking $\tau_0$ sufficiently small and using the discrete Gronwall's inequality, we get
\begin{equation}\label{ccff10}
\left\|\mathbf{e}^{m}\right\|^2_{l^{2}}
\lesssim \left(\frac{h^{2}}{\varepsilon}+\frac{\tau^{2}}{\varepsilon^{3}}\right)^2, \quad 1 \leq m \leq \frac{T}{\tau},
\end{equation}
which directly demonstrates the error estimate \eqref{tccnfd}.
\end{proof}
Actually, in the massless and nonrelativistic regime, based on  Theorem \ref{4th4}, when given an accuracy bound $\delta>0,$ the $\varepsilon$-scalability (or resolution) of the CNFD method is:
\begin{equation}\label{g1}
h=O(\sqrt{\delta \varepsilon})=O(\sqrt{\varepsilon}), \quad \tau=O\left(\sqrt{\delta\varepsilon^{3}}\right)=O\left(\sqrt{\varepsilon^{3}}\right), \quad 0<\varepsilon \ll 1.
\end{equation}

Furthermore, we get the following error estimates of the the total density and current density for the CNFD method.
%which omit for simplicity in the following analysis.

\begin{cor}\label{1co1}
Under the assumptions in (A) and (B), there exist the constants $h_0>0$, $\tau_0>0$ independent of $\varepsilon$ and sufficiently small, such that for any $0<\varepsilon\le 1$, if $\,0<h\le h_0$ and $\,0<\tau\le\tau_0$, for the CNFD method \eqref{fuz4}, we obtain the error estimate on the total density as follow
\begin{equation}\label{colp}
\left\|\rho^n-\rho(t_n,\cdot)\right\|_{l^{2}} \lesssim \frac{h^{2}}{\varepsilon}+\frac{\tau^{2}}{\varepsilon^{3}}, \qquad 0 \leq n \leq \frac{T}{\tau},
\end{equation}
in which $\rho^n$ is from the wave function $\varPhi^n$ in \eqref{pas} with $d=1$.
\end{cor}

\begin{cor}\label{1co2}
Under the assumptions in (A) and (B), there exist the constants $h_0>0$, $\tau_0>0$ independent of $\varepsilon$ and sufficiently small, such that for any $0<\varepsilon\le1$, if $\,0<h\le h_0$ and $\,0<\tau\le\tau_0$, for the CNFD method \eqref{fuz4}, we obtain the error estimate on the current density as follow
\begin{equation}\label{colJ}
\left\|\textrm{\emph{\textbf{J}}}^n-\emph{\textbf{J}}(t_n,\cdot)\right\|_{l^{2}} \lesssim \frac{h^{2}}{\varepsilon}+\frac{\tau^{2}}{\varepsilon^{3}}, \qquad 0 \leq n \leq \frac{T}{\tau},
\end{equation}
in which $\emph{\textbf{J}}^\textrm{n}$ is from the wave function $\varPhi^n$ in \eqref{cde} with $d=1$.
\end{cor}
We remark here that Corollaries \ref{1co1}, \ref{1co2} also hold for SIFD1, LFFD and SIFD2 methods introduced later, if the corresponding stability conditions are satisfied.

\subsection{Numerical results}\label{cnfdnr}
In the discussion below, we numerically study the temporal and spatial scalability of the CNFD method for the 1D Dirac equation \eqref{x1} in the massless and nonrelativistic regime. The Dirac equation is solved on a bounded domain $\Omega=(-1,1)$ with periodic boundary conditions on $\partial\Omega$. Here the `reference exact' solution $\varPhi(t,x)=(\varPhi_1(t,x),\varPhi_2(t,x))^T$ is obtained by using the time-splitting Fourier pseudospectral method with a very small time step $\tau_e=10^{-6}$ and a very fine mesh size $h_e=1/16384$ respectively so that the errors in corresponding directions are sufficiently small. In the following example, we choose the electric and magnetic potential as
\begin{equation}\label{nte}
V(t, x)=\frac{1}{2+{\textrm{sin}(\pi x)}}, \quad A_{1}(t, x)=\frac{1}{1+{\textrm{cos}^{2}(\pi x)}},  \quad x \in \Omega, \, t \geq 0,
\end{equation}
with the initial condition as
\begin{equation}\label{nint}
\varPhi_{1}(0, x)=\textrm{sin}(\pi x)+\textrm{sin}(2\pi x), \quad \varPhi_{2}(0, x)=\textrm{cos}(\pi x), \quad x \in \Omega.
\end{equation}

In order to quantify the numerical errors of the finite difference methods for the Dirac equation, we give the error expressions of the wave function $\varPhi$, the total density $\rho$ and the current density $\textbf{J}$ as follows

\begin{eqnarray}\label{abe}
 \begin{aligned}
&{e^{h,\tau}_\varPhi(t_n)}=\big\| \varPhi^n-\varPhi(t_n,\cdot)\big\|_{l^2}=\sqrt{h\sum_{j=0}^{M-1}
|\varPhi_j^n-\varPhi(t_n,x_j)|^2}, \\
&e^{h,\tau}_\rho(t_n)=\big\| \rho^n-\rho(t_n,\cdot)\big\|_{l^1}=h\sum_{j=0}^{M-1}
|\rho_j^n-\rho(t_n,x_j)|,\\
&e^{h,\tau}_\textbf{J}(t_n)=\frac{\big\| \textbf{J}^n-\textbf{J}(t_n,\cdot)\big\|_{l^1}}
{\big\| \textbf{J}(t_n,\cdot) \big\|_{l^1}}
=\frac{{\sum\limits_{j=0}^{M-1}
|\textbf{J}_j^n-\textbf{J}(t_n,x_j)}|}{{\sum\limits_{j=0}^{M-1}
|\textbf{J}(t_n,x_j)|}},
 \end{aligned}
\end{eqnarray}
in which $e_\cdot^h(t_n),\,e_\cdot^\tau(t_n)$ are denoted as the spatial and temporal errors, respectively. Here $\rho^n$ and $\textbf{J}^n$ can be obtained by the numerical solution of $\varPhi$ in view of the definition in \eqref{pas} and \eqref{cde}.

Table \ref{cnfdht} displays spatial errors $e_\varPhi^h(t=2)$ (upper) and temporal errors $e_\varPhi^\tau(t=2)$ (lower) of the wave function respectively with different mesh size $h$ and time step $\tau$ for the CNFD method \eqref{fuz4}. From Table \ref{cnfdht}, for any $\varepsilon\in(0,1]$, we can directly observe that the CNFD method \eqref{fuz4} has second order convergence in both time and space.

\begin{table}[H]%[htp]
 \caption{{\normalsize Spatial and temporal error analysis of the wave function $e_\varPhi^{h,\tau}(t=2)$ for the CNFD method}}
 \vspace{0.8mm}\label{cnfdht}
 \centering
 \setlength{\tabcolsep}{5.5mm}%{5mm}
 %\resizebox{\textwidth}{15mm}
 {\begin{tabular}{cccccc}
 \hline
$e_\varPhi^h(t=2)$ &$h_0=1/16$ &$h_0/2$  &$h_0/2^2$  &$h_0/2^3$ &$h_0/2^4$\\
 \hline\Xhline{0.8pt}
$\varepsilon_0=1$&\textbf{3.35E-1}&8.48E-2&2.12E-2&5.30E-3&1.33E-3\\
Order&$-$&1.98&2.00&2.00&2.00\\
$\varepsilon_0/4$&1.20&\textbf{3.22E-1}&8.11E-2&2.03E-2&5.07E-3\\
Order&$-$&\textbf{1.90}&1.99&2.00&2.00\\
$\varepsilon_0/4^{2}$&1.43&1.21&\textbf{3.22E-1}&8.09E-2&2.02E-2\\
Order&$-$&0.24&\textbf{1.91}&1.99&2.00\\
$\varepsilon_0/4^{3}$&2.96&1.41&1.21&\textbf{3.22E-1}&8.09E-2\\
Order&$-$&1.07&0.22&\textbf{1.91}&1.99\\
\hline
$e_\varPhi^\tau(t=2)$ &$\tau_0=1/40$ &$\tau_0/4$  &$\tau_0/4^2$  &$\tau_0/4^3$&$\tau_0/4^4$\\
\hline\Xhline{0.8pt}
$\varepsilon_0=1$&\textbf{3.44E-2}&2.16E-3&1.35E-4&8.75E-6&6.37E-7\\
Order&$-$&2.00&2.00&1.98&1.89\\
$\varepsilon_0/4^{2/3}$&4.45E-1&\textbf{2.87E-2}&1.80E-3&1.13E-4&7.80E-6\\
Order&$-$&\textbf{1.98}&2.00&2.00&1.93\\
$\varepsilon_0/4^{4/3}$&1.34&4.35E-1&\textbf{2.77E-2}&1.73E-3&1.10E-4\\
Order&$-$&0.81&\textbf{1.99}&2.00&1.99 \\
$\varepsilon_0/4^2$&1.87&1.31&4.34E-1&\textbf{2.74E-2}&1.79E-3\\
Order&$-$&0.25&0.80&\textbf{1.99}&1.97\\
\hline
\end{tabular}}
\end{table}

For the discretization error in space, the upper triangle above the the bold diagonal line in the top half of Table \ref{cnfdht} indicates that second order convergence exists for the CNFD method when $h=O(\varepsilon^{1/2})$. Similarly, for the discretization error in time, the CNFD method has second order convergence only when $\tau=O(\varepsilon^{3/2})$, which is verified through the upper triangle above the bold diagonal line in the bottom half of Table \ref{cnfdht}. Hence, in the massless and nonrelativistic regime, the $\varepsilon$-resolution for the CNFD method is $h=O(\varepsilon^{1/2})$ on mesh size and $\tau=O(\varepsilon^{3/2})$ on time step, which is consistent with our error estimates in Theorem \ref{4th4}.

In the following sections, for the Dirac equation \eqref{x1} in the massless and nonrelativistic regime, we will introduce another three finite difference methods including the leap-frog and two semi-implicit finite difference methods, and take the one semi-implicit method as an example to verify the conclusions in Corollaries \ref{1co1} and \ref{1co2}.

\section{A semi-implicit finite difference (SIFD1) method and its error estimate}
In this section, we propose the semi-implicit finite difference (SIFD1) method for \eqref{x1}-\eqref{x2} in which we adopt explicit discretization for the differential term and implicit discretization for the rest terms. Compared to the CNFD method in the previous section, the advantage of this scheme is that there is no need to solve coupled linear systems and thus it is more efficient.
\subsection{The SIFD1 method}
We consider the semi-implicit finite difference (SIFD1) scheme to discretize the equation \eqref{x1} for $n\ge 1,\,j=0,1,\cdot\cdot\cdot,M-1,$
\begin{equation}\label{fuz2}
 i \delta_{t} \varPhi_{j}^{n}=-\frac{i}{\varepsilon}\sigma_{1} \delta_{x} \varPhi_{j}^{n}+\left(\frac{\sigma_{3}}{\varepsilon}+V_{j}^{n} I_{2}-A_{1, j}^{n} \sigma_{1}\right) \frac{\varPhi_{j}^{n+1}+\varPhi_{j}^{n-1}}{2}.
\end{equation}
Its discrete boundary and initial conditions are the same as \eqref{zbcn}. By applying Taylor expansion and noticing the Dirac equation \eqref{x1}, the first step for the SIFD1 method \eqref{fuz2} can be designed as
\begin{equation}\label{zf1si1}
\varPhi_{j}^{1}=\varPhi_{j}^{0}-\sin \left(\frac{\tau}{\varepsilon}\right) \sigma_{1} \varPhi_{0}^{\prime}\left(x_{j}\right)-i\left(\sin \left(\frac{\tau}{\varepsilon}\right) \sigma_{3}+\tau V_{j}^{0} I_{2}-\tau A_{1, j}^{0} \sigma_{1}\right) \varPhi_{j}^{0},\quad j=0,1, \ldots, M,
\end{equation}
in which we adopt $\frac{1}{\tau}\textrm{sin}(\frac{\tau}{\varepsilon})$ instead of $\frac{1}{\varepsilon}$ such that \eqref{zf1si1} have second order convergence with $\tau$ for any fixed $0<\varepsilon\le1$ and $\|\varPhi^1\|_{l^{\infty}}=\mathop {\max }\limits_{0 \le j \le M } \left| {{\varPhi_j^1}} \right|\lesssim 1$ for any $0<\varepsilon\le1.$ Here we remark when $\varepsilon=1,$ it can be replaced by $1$.

We notice that the SIFD1 method is time symmetric, in other words, it is unchanged under $n+1\leftrightarrow n-1$ and $\tau\leftrightarrow-\tau$, and its memory cost is $O(M)$. Then the SIFD1 method \eqref{fuz2} is implicit, but for every time step of $n\ge1$, its corresponding linear system is decoupled, as well as it can be solved explicitly as below
\begin{equation*}
  \varPhi_j^{n+1}=\left\{(i-\tau V_j^n)I_2-\frac{\tau}{\varepsilon}\sigma_3+\tau A_{1,j}^n\sigma_1   \right\}^{-1}H_j^n,\qquad j=0,1,\ldots,M-1,
\end{equation*}
in which $H_j^n=\left\{\left((i+\tau V_j^n)I_2+
\frac{\tau}{\varepsilon}\sigma_3  -\tau A_{1,j}^n\sigma_1\right)\varPhi_j^{n-1}-\frac{2i\tau}{\varepsilon}
\sigma_1\delta_x\varPhi_j^n \right\}$. Thus, the computational cost of SIFD1 method per step also is $O(M)$.
\subsection{Linear stability analysis}
For any $U\in X_M,$ we denote the corresponding Fourier representation as
\begin{equation}\label{xx}
U_{j}=\sum_{l=-M / 2}^{M / 2-1} \widetilde{U}_{l} e^{i \mu_{l}\left(x_{j}-a\right)}=\sum_{l=-M / 2}^{M / 2-1} \widetilde{U}_{l} e^{2 i j l \pi / M}, \quad j=0,1, \ldots, M,
\end{equation}
in which $\mu_l$ and $\widetilde{U}_l\in\mathbb{C}^2$ are defined as
\begin{equation}\label{y}
\mu_{l}=\frac{2 l \pi}{b-a}, \quad \widetilde{U}_{l}=\frac{1}{M} \sum_{j=0}^{M-1} U_{j} e^{-2 i j l \pi / M}, \quad l=-\frac{M}{2}, \ldots, \frac{M}{2}-1.
\end{equation}
\begin{lemma}\label{lemsifd1}
The SIFD1 method \eqref{fuz2} is stable under its corresponding stability condition
\begin{equation}\label{Sta2}
0<\tau \leq \varepsilon h, \qquad h>0, ~ 0<\varepsilon \leq 1.
\end{equation}
\end{lemma}
\begin{proof}
Due to the fact that the implicit part $\left(\sigma_{3}/\varepsilon+V_{j}^{n} I_{2}-A_{1, j}^{n} \sigma_{1}\right) \frac{\varPhi_{j}^{n+1}+\varPhi_{j}^{n-1}}{2}$ is automatically stable, here we just need to concentrate on the explicit part $i\delta_t\varPhi_j^n=-\frac{i}{\varepsilon}\sigma_1\delta_x\varPhi_j^n.$ Plugging
\begin{equation}\label{prlefd1}
\varPhi_{j}^{n}=\sum_{l=-M / 2}^{M / 2-1} \xi_{l}^{n} \widetilde{\left(\varPhi^{0}\right)}_{l} e^{i \mu_{l}\left(x_{j}-a\right)}=\sum_{l=-M / 2}^{M / 2-1} \xi_{l}^{n} \widetilde{\left(\varPhi^{0}\right)}_{l} e^{2 i j l \pi / M},\quad 0 \leq j \leq M,
\end{equation}
with $\xi_l^n\in \mathbb{C}$ being the amplification factor of the $l$-th mode in the phase space and $\widetilde{\left(\varPhi^{0}\right)}_{l}$ being the Fourier coefficient at $n=0$. Plug \eqref{prlefd1} into the SIFD1 method \eqref{fuz2}, we obtain the corresponding amplification factor $\xi_l$ satisfies
\begin{equation}\label{prlefdX1}
\xi_{l}^{2}-2 i \tau \theta_{l} \xi_{l}-1=0, \quad l=-\frac{M}{2}, \ldots, \frac{M}{2}-1,
\end{equation}
in which $\theta_{l}=\pm\frac{sin(\mu_lh)}{\varepsilon h}$. Hence the stability condition is equivalent to
\begin{equation}\label{prlefdX2}
\left|\xi_{l}\right| \leq 1 \Longleftrightarrow\left|\tau\theta_{l}\right| \leq 1, \quad l=-\frac{M}{2}, \ldots, \frac{M}{2}-1,
\end{equation}
which means $\left|\dfrac{\tau}{\varepsilon h}\right| \leq 1$, and gives $0 <\tau \leq \varepsilon h$.
\end{proof}

\subsection{Error estimate}
The error estimate for SIFD1 is given as follows.
\begin{theorem}\label{2th2}
Under the assumptions in (A) and (B), there exist the constants $h_0>0,\,\tau_0>0$ independent of $\varepsilon$ and sufficiently small, such that for any $0<\varepsilon\le1,$ if $\,0<h \le h_0,\,0<\tau\le \tau_0$ and under the stability condition \eqref{Sta2}, for the SIFD1 method \eqref{fuz2} with \eqref{zbcn} and \eqref{zf1si1}, we obtain the error estimate on the wave function as below
\begin{equation}\label{tcsifd1}
\left\|\mathrm{\textbf{e}}^{n}\right\|_{l^{2}} \lesssim \frac{h^{2}}{\varepsilon}+\frac{\tau^{2}}{\varepsilon^{3}}, \qquad 0 \leq n \leq \frac{T}{\tau}.
\end{equation}
\end{theorem}
\begin{proof}
The local truncation error $\eta^n=(\eta_0^n,{\eta}_1^n,\ldots,
{\eta}_M^n)^T\in X_M$ of the SIFD1 \eqref{fuz2} with \eqref{zbcn} and \eqref{zf1si1} for $0\le j\le M-1$ and $n\ge1$ is defined as follows
\begin{eqnarray}\label{cfth1}
{\eta}_{j}^{0}&:=&i\left( \delta_{t}^{+} +\frac{1}{\varepsilon} \sigma_{1} \delta_{x} \right)\varPhi_{0}\left(x_{j}\right)
-\left(\frac{\sigma_{3}}{\varepsilon} +V_{j}^{0} I_{2}-A_{1, j}^{0} \sigma_{1}\right) \varPhi_{0}\left(x_{j}\right),\\\label{cfth2}
{\eta}_{j}^{n}&:=&i\left( \delta_{t} +\frac{1}{\varepsilon} \sigma_{1} \delta_{x} \right)\varPhi(t_n,x_j)
-\left(\frac{\sigma_{3}}{\varepsilon}  +V_{j}^{n} I_{2}-A_{1, j}^{n} \sigma_{1}\right) \frac{\varPhi(t_{n+1},x_j)+\varPhi(t_{n-1},x_j)}{2}.
\end{eqnarray}
By applying the Taylor expansion to \eqref{cfth1} and \eqref{cfth2}, we get for $j=0,1,\cdot\cdot\cdot,M-1$ and $n\ge1,$
\begin{equation}\label{cf2}
\begin{aligned}
&{\eta}_{j}^{0}=\frac{i \tau}{2} \partial_{t t} \varPhi\left(\tau^{\prime}, x_{j}\right)
+\frac{ih^2}{6\varepsilon} \sigma_1 \partial_{x x x} \varPhi_{0}\left(x_j^{\prime}\right),\\
&{\eta}_{j}^{n}=\frac{i \tau^2}{6} \partial_{t t t} \varPhi\left(t_{n}^{\prime}, x_{j}\right)+\frac{i h^{2}}{6\varepsilon} \sigma_{1} \partial_{x x x} \varPhi\left(t_{n}, x_{j}^{\prime}\right)-\frac{\tau^{2}}{2 }\Big(\frac{1}{\varepsilon}\sigma_{3}+V_{j}^{n} I_{2}-A_{1, j}^{n} \sigma_{1}\Big) \partial_{t t} \varPhi\left(t_{n}^{\prime \prime}, x_{j}\right),
\end{aligned}
\end{equation}
in which $\tau'\in(0,\tau),\,t_n',t_n{''}\in(t_{n-1},t_{n+1})$ and $x'_j\in (x_{j-1},x_{j+1})$.
Noticing \eqref{x1} and the assumptions in (A) and (B), we obtian
\begin{equation}\label{cf3}
 \left|{\eta}_{j}^{0}\right| \lesssim
 \frac{h^2}{\varepsilon}+\frac{\tau}{\varepsilon^2}, \qquad
 \left|{\eta}_{j}^{n}\right| \lesssim
 \frac{h^2}{\varepsilon}+\frac{\tau^2}{\varepsilon^3}, \quad j=0,1, \ldots, M-1, \, n \geq 1,
\end{equation}
which directly implies
\begin{equation}\label{cf4}
\left\|{\eta}^{n}\right\|_{l^{\infty}}=\max _{0 \leq j \leq M-1}\left|{\eta}_{j}^{n}\right| \lesssim \frac{h^2}{\varepsilon}+\frac{\tau^2}{\varepsilon^3},\qquad
\left\|{\eta}^{n}\right\|_{l^{2}} \lesssim\left\|{\eta}^{n}\right\|_{l^{\infty}} \lesssim \frac{h^{2}}{\varepsilon}+\frac{\tau^2}{\varepsilon^3}, \quad n \geq 1.
\end{equation}
Subtracting \eqref{fuz2} from \eqref{cfth2} and noticing \eqref{errfun}, we obtain the error function with $0\le j\le M-1$ and $n\ge1$ as below
\begin{equation}\label{cf41}
 i \delta_{t} \mathbf{e}_{j}^{n}=-\frac{i}{\varepsilon} \sigma_{1} \delta_{x} \mathbf{e}_{j}^{n}+
\frac{1}{2}\left(\frac{\sigma_{3}}{\varepsilon} +V_{j}^{n} I_{2}-A_{1, j}^{n} \sigma_{1}\right) (\mathbf{e}_{j}^{n+1}+\mathbf{e}_{j}^{n-1})+{\eta}_{j}^{n},
\end{equation}
in which its initial and boundary conditions \eqref{cf5} are the same as given before. For the first step, we obtain
\begin{equation}\label{cf42}
\left\|\mathbf{e}^{1}\right\|_{l^{2}}=\tau\left\|{\eta}^{0}\right\|_{l^{2}} \lesssim
\tau\left(\frac{ h^{2}}{\varepsilon}+\frac{\tau}{\varepsilon^2} \right)\lesssim \frac{h^{2}}{\varepsilon}+\frac{\tau^{2}}{\varepsilon^3},
\end{equation}
Denote $\mathcal{E}^{n+1}$ as
\begin{equation}\label{cf6}
\mathcal{E}^{n+1}=\left\|\mathbf{e}^{n+1}\right\|_{l^{2}}^{2}+\left\|\mathbf{e}^{n}\right\|_{l^{2}}^{2}+2 \operatorname{Re}\bigg(\frac{h\tau}{\varepsilon} \sum_{j=0}^{M-1}\left(\mathbf{e}_{j}^{n+1}\right)^{*} \sigma_{1} \delta_{x} \mathbf{e}_{j}^{n}\bigg),\quad n\ge0.
\end{equation}
and under its stability condition of \eqref{Sta2}, i.e. $0<\tau \leq \varepsilon h \tau_1$ with $\tau_1=\frac{1}{2}$, which implies $\frac{\tau}{\varepsilon h}\le\frac{1}{2}$, by using the Cauchy inequality, we could derive
\begin{equation}\label{cf7}
\frac{1}{2}\left(\left\|\mathbf{e}^{n+1}\right\|_{l^{2}}^{2}+\left\|\mathbf{e}^{n}\right\|_{l^{2}}^{2}\right) \leq \mathcal{E}^{n+1} \leq \frac{3}{2}\left(\left\|\mathbf{e}^{n+1}\right\|_{l^{2}}^{2}+\left\|\mathbf{e}^{n}\right\|_{l^{2}}^{2}\right), \quad n \geq 0.
\end{equation}
From \eqref{cf42}, we have
\begin{equation}\label{epes1}
\mathcal{E}^1 \lesssim \bigg(\frac{h^2}{\varepsilon}+\frac{\tau^2}{\varepsilon^3} \bigg)^2,
\end{equation}
Multiplying $2h\tau(\mathbf{e}_j^{n+1}+\mathbf{e}_j^{n-1})^*$ from the left on both side to \eqref{cf41}, by taking its imaginary part, summing up the equation for $j=0,1,\cdot\cdot\cdot,M-1,$ and using the Cauchy inequality as before, then noticing \eqref{cf4} and \eqref{cf7}, we obtain for $n\ge1,$
\begin{equation}\label{cf8}
 \begin{aligned}
 \mathcal{E}^{n+1}-\mathcal{E}^{n}
 & =2h\tau \operatorname{Im}\bigg( \sum_{j=0}^{M-1}\left(\mathbf{e}_{j}^{n+1}+\mathbf{e}_{j}^{n-1}\right)^{*} {{\eta}}_{j}^{n}\bigg)\\
 & \lesssim \tau\Big(\mathcal{E}^{n}+\mathcal{E}^{n+1}\Big)
 +\tau\left(\frac{h^2}{\varepsilon}+\frac{\tau^2}{\varepsilon^3}\right)^{2}, \quad n \geq 0,
 \end{aligned}
\end{equation}
Summing up the above inequality of \eqref{cf8} for $n=1,2,\cdot\cdot\cdot,m-1,$ we obtain that
\begin{equation}\label{cf9}
\mathcal{E}^{m}-\mathcal{E}^{1} \lesssim \tau \sum_{s=1}^{m} \mathcal{E}^{s}+m \tau\left(\frac{h^{2}}{\varepsilon}+\frac{\tau^2}{\varepsilon^3}\right)^{2}, \quad 1 \leq m \leq \frac{T}{\tau}.
\end{equation}
Hence if we take $\tau_0$ sufficiently small, use the discrete Gronwall's inequality, and notice the inequality \eqref{epes1}, we get
\begin{equation}\label{cf10}
\mathcal{E}^{m} \lesssim\left(\frac{h^{2}}{\varepsilon}+\frac{\tau^{2}}{\varepsilon^{3}}\right)^{2}, \quad 1 \leq m \leq \frac{T}{\tau},
\end{equation}
which directly demonstrates the error estimate \eqref{tcsifd1} in view of \eqref{cf7}.
\end{proof}

From Theorem \ref{2th2}, in the massless and nonrelativistic regime, when given an accuracy bound $\delta>0,$ the $\varepsilon$-resolution of the SIFD1 method is:
\begin{equation}\label{g1}
h=O(\sqrt{\delta \varepsilon})=O(\sqrt{\varepsilon}), \quad
\tau=O\left(\sqrt{\delta\varepsilon^{3}}\right)=O\left(\sqrt{\varepsilon^{3}}\right), \quad 0<\varepsilon \ll 1.
\end{equation}

\subsection{Numerical results}
In the following numerical simulation, the electromagnetic potential, initial condition, error functions and mesh sizes are same as \eqref{nte}-\eqref{abe} in Subsection \ref{cnfdnr}. In order to satisfy its stability condition and accuracy requirement for the SIFD1 methods, in Table \ref{sifd1ht}, we take
\begin{equation}\label{stacop}
\delta_j(\varepsilon)=\frac{1}{2^k},\qquad (\varepsilon=\frac{1}{4^{2k/3}},\quad k=0,1,\ldots,\, j=1,2,\ldots)
\end{equation}

\begin{table}[H]%[h!]%[htp]
 \caption{{\normalsize Spatial and temporal error analysis of the wave function $e_\varPhi^{h,\tau}(t=2)$ for the SIFD1 method}}
 \vspace{0.8mm}\label{sifd1ht}
 \centering
 \setlength{\tabcolsep}{5.2mm}%{5mm}
 %\resizebox{\textwidth}{15mm}
 {\begin{tabular}{cccccccc}
 \hline
$e_\varPhi^h(t=2)$ &$h_0=1/16$ &$h_0/2$  &$h_0/2^2$  &$h_0/2^3$ &$h_0/2^4$\\
 \hline\Xhline{0.8pt}
$\varepsilon_0=1$&\textbf{3.35E-1}&8.48E-2&2.12E-2&5.30E-3&1.33E-3\\
Order&$-$&1.98&2.00&2.00&2.00\\
$\varepsilon_0/4$&1.20&\textbf{3.22E-1}&8.11E-2&2.03E-2&5.07E-3\\
Order&$-$&\textbf{1.90}&1.99&2.00&2.00\\
$\varepsilon_0/4^{2}$&1.43&1.21&\textbf{3.22E-1}&8.08E-2&2.02E-2\\
Order&$-$&0.24&\textbf{1.91}&1.99&2.00\\
$\varepsilon_0/4^{3}$&2.96&1.41&1.21&\textbf{3.20E-1}&7.88E-2\\
Order&$-$&1.07&0.22&\textbf{1.92}&2.02\\
\hline
\multirow{2}*{$e_\Phi^\tau(t=2)$} &$\tau_0=1/40$ &$\tau_0/4$  &$\tau_0/4^2$  &$\tau_0/4^3$&$\tau_0/4^4$\\
&$h_0=1/16$ &$h_0/4\delta_1(\varepsilon)$ &$h_0/4^2\delta_2(\varepsilon)$ &$h_0/4^3\delta_3(\varepsilon)$&$h_0/4^4\delta_4(\varepsilon)$\\
\hline\Xhline{0.8pt}
$\varepsilon_0=1$&\textbf{2.93E-1}&1.81E-2&1.13E-3&7.05E-5&4.40E-6\\
Order&$-$&2.01&2.00&2.00&2.00\\
$\varepsilon_0/4^{2/3}$&Unstable&\textbf{1.53E-1}&9.56E-3&5.98E-4&3.74E-5\\
Order&$-$&$-$&2.00&2.00&2.00\\
$\varepsilon_0/4^{4/3}$&Unstable&1.19&\textbf{7.93E-2}&4.95E-3&3.10E-4\\
Order&$-$&$-$&\textbf{1.96}&2.00&2.00\\
$\varepsilon_0/4^{2}$&Unstable&2.44&4.97E-1&\textbf{3.11E-2}&1.94E-3\\
Order&$-$&$-$&1.15&\textbf{2.00}&2.00\\
\hline
\end{tabular}}
\end{table}

Table \ref{sifd1ht} presents the spatial errors $e_\varPhi^h(t=2)$ and temporal errors $e_\varPhi^\tau(t=2)$ of the wave function with various mesh sizes by using the SIFD1 method \eqref{fuz2}. From Table \ref{sifd1ht}, for any $\varepsilon\in(0,1]$, we can observe that the SIFD1 method \eqref{fuz2} has second order convergence in space and time.  The $\varepsilon$-resolution of the SIFD1 method is still $h=O(\varepsilon^{1/2})$ and $\tau=O(\varepsilon^{3/2})$, which is verified by the upper triangles above the diagonal lines labelled with bold type in the top and bottom half of the table. Numerical results correspond well with our error estimate in Theorem \ref{2th2}.

\section{Other finite difference methods and their error estimates}
Here we propose another semi-implicit finite difference (SIFD2) scheme and the explicit leap-frog finite difference (LFFD) scheme, and establish their error estimates.
\subsection{The SIFD2 method and LFFD method}
We consider the two other finite difference methods to discretize the Dirac equation \eqref{x1} for $n\ge 1,\,j=0,1,\cdot\cdot\cdot,M-1.$ Another semi-implicit finite difference (SIFD2) scheme is given as follows
\begin{equation}\label{fuz3}
i \delta_{t} \varPhi_{j}^{n}=\frac{1}{\varepsilon}\Big(-i \sigma_{1} \delta_{x}+\sigma_{3}\Big) \frac{\varPhi_{j}^{n+1}+\varPhi_{j}^{n-1}}{2}+\Big(V_{j}^{n} I_{2}-A_{1, j}^{n} \sigma_{1}\Big) \varPhi_{j}^{n},
\end{equation}
and the leap-frog finite difference (LFFD) scheme is
\begin{equation}\label{fuz1}
i \delta_{t} \varPhi_{j}^{n}=\frac{1}{\varepsilon}\Big(-i \sigma_{1} \delta_{x} +\sigma_{3}\Big)\varPhi_{j}^{n}+\left(V_{j}^{n} I_{2}-A_{1, j}^{n} \sigma_{1}\right) \varPhi_{j}^{n}.
\end{equation}
Their discrete boundary and initial conditions are the same as \eqref{zbcn}, and the first steps for SIFD2 \eqref{fuz3}, LFFD \eqref{fuz1} are similar to SIFD1 in \eqref{zf1si1}.

Here we notice that the SIFD2 and LFFD methods are time symmetric, in other words, they remain unchanged under $n+1\leftrightarrow n-1,\,\tau\leftrightarrow-\tau$, and their memory costs are both $O(M)$. The SIFD2 scheme \eqref{fuz3} is implicit, which means that at each time step for $n\ge1$, its corresponding linear system can be decoupled in phase (Fourier) space, as well as it can be solved explicitly in phase space as following
\begin{equation*}
\widetilde{(\varPhi^{n+1})}_l=\left(iI_2-\frac{\tau \textrm{sin}(\mu_lh)}{\varepsilon h}\sigma_1-\frac{\tau}{\varepsilon}\sigma_3 \right)^{-1}L_l^n,\qquad l=-M/2,\ldots,M/2-1,
\end{equation*}
in which
\begin{equation*}
  L_l^n=\left\{\left(iI_2+ \frac{\tau \textrm{sin}(\mu_lh)}{\varepsilon h}\sigma_1+\frac{\tau}{\varepsilon}\sigma_3\right)\widetilde
  {(\varPhi^{n-1})}_l+2\tau\widetilde{(G^n\varPhi^n)}_l\right\},
\end{equation*}
and $G^n=(G_0^n,G_1^n,\ldots,G_M^n)^T\in X_M$ with $ G_j^n=V_j^nI_2-A_{1,j}^n\sigma_1$ for $j=0,1,\ldots,M,$ and hence its computational cost per step is $O(M\,\textrm{ln}\,M)$. The LFFD method \eqref{fuz1} is explicit and its computational cost per step is $O(M)$. When $\varepsilon=1$, it should be the most efficient and simplest method for the Dirac equation \eqref{maiequ} and thus the LFFD method has been widely used. From what has been analysed above on the computational cost per time step, we can conclude that the CNFD method is the most expensive one and the LFFD method is the most efficient among the four finite difference methods.

\subsection{Linear stability analysis}
In the following, in order to realize the linear stability analysis of the finite difference methods for the Dirac equation \eqref{x1} through the von Neumann method \cite{S85}, we assume that  $V(t,x)\equiv V^0$
and $A_1(t,x)\equiv A_1^0$ with $V^0$ and $A_1^0$ being two real constants. Next we have the following conclusions from the SIFD2 and LFFD methods.
\begin{lemma}\label{lemsifd2}
The SIFD2 method \eqref{fuz3} is stable under its corresponding stability condition
\begin{equation}\label{Sta3}
0<\tau \leq \frac{1}{\left|V^{0}\right|+\left|A_{1}^{0}\right|}, \qquad h>0, ~ 0<\varepsilon \leq 1.
\end{equation}
\end{lemma}
\begin{proof}
Reference to the proof of  Lemma \ref{lemsifd1}, here we only need to focus on the explicit part $i\delta_t\varPhi_j^n=\big(V_j^0I_2-A_{i,j}^0\sigma_1\big)\varPhi_j^n.$
Substituting the formula \eqref{prlefd1} into the explicit part of \eqref{fuz3}, we get
\begin{equation}\label{prlefdX3}
\xi_{l}^{2}-2 i \tau (-V^0\pm A_1^0) \xi_{l}-1=0, \quad l=-\frac{M}{2}, \ldots, \frac{M}{2}-1,
\end{equation}
which indicates the stability condition is $0<\tau\le\frac{1}{|V^0|+|A_1^0|}.$
\end{proof}

\begin{lemma}\label{lem}
The LFFD method \eqref{fuz1} is stable under its corresponding stability condition
\begin{equation}\label{Sta1}
0<\tau \leq \frac{\varepsilon h}{\left|V^{0}\right| \varepsilon h+\sqrt{h^2+\left(\left|A_{1}^{0}\right|\varepsilon h+1\right)^{2}}}, \qquad h>0, ~ 0<\varepsilon \leq 1.
\end{equation}
\end{lemma}
\begin{proof}
(\romannumeral1) Plugging \eqref{prlefd1} into the LFFD method, and considering the orthogonality for the Fourier series, we have
\begin{equation}\label{prlefd2}
\left|\left(\xi_{l}^{2}-1\right) I_{2}+2 i \tau \xi_{l}\left(\frac{\sigma_{3}}{\varepsilon} +V^{0} I_{2}-A_{1}^{0} \sigma_{1}+\frac{\sin \left(\mu_{l} h\right)}{\varepsilon h} \sigma_{1}\right)\right|=0,\quad l=-\frac{M}{2},\cdot\cdot\cdot,\frac{M}{2}-1,
\end{equation}
Substituting \eqref{pau} into \eqref{prlefd2}, we obtain the amplification factor $\xi_l$ satisfies
\begin{equation}\label{prlefd3}
\xi_{l}^{2}-2 i \tau \theta_{l} \xi_{l}-1=0, \quad l=-\frac{M}{2}, \ldots, \frac{M}{2}-1,
\end{equation}
in which
\begin{equation}\label{prlefd4}
\theta_{l}=-V^{0} \pm \frac{1}{\varepsilon h} \sqrt{h^{2}+\left(-A_{1}^{0} \varepsilon h+\sin \left(\mu_{l} h\right)\right)^{2}}, \quad l=-\frac{M}{2}, \ldots, \frac{M}{2}-1,
\end{equation}
and the stability condition for the LFFD method \eqref{fuz1} equivalents to
\begin{equation}\label{prlefd5}
\left|\xi_{l}\right| \leq 1 \Longleftrightarrow\left|\tau\theta_{l}\right| \leq 1, \quad l=-\frac{M}{2}, \ldots, \frac{M}{2}-1,
\end{equation}
hence it directly gives the condition \eqref{Sta1}.
\end{proof}
\subsection{Error estimate}
According to the assumption (B) in \eqref{lefs2}, the stability condition of the SIFD2 method becomes
\begin{equation}\label{lefs3}
0<\tau \leq \frac{1}{ V_{\max }+A_{1, \max }}, \qquad h>0, ~ 0<\varepsilon \leq 1,
\end{equation}
the stability condition of the LFFD method becomes
\begin{equation}\label{lefs1}
0<\tau \leq \frac{\varepsilon h}{\varepsilon h V_{\max }+\sqrt{h^{2}+\left(\varepsilon h A_{1, \max }+1\right)^{2}}}, \qquad h>0, ~ 0<\varepsilon \leq 1,
\end{equation}
next we could establish the error bounds under these stability conditions.
\begin{theorem}\label{3th3}
Under the assumptions in (A) and (B), there exist the constants $h_0>0,\,\tau_0>0$ independent of $\varepsilon$ and sufficiently small, such that for any $0<\varepsilon\le1,$ if $\,0<h \le h_0,\,0<\tau\le \tau_0$ and under the above stability condition \eqref{lefs3}, for the SIFD2 method \eqref{fuz3} with \eqref{zbcn} and \eqref{zf1si1}, we obtain the error estimate on the wave function as below
\begin{equation}\label{tcsifd2}
\left\|\mathrm{\textbf{e}}^{n}\right\|_{l^{2}} \lesssim \frac{h^{2}}{\varepsilon}+\frac{\tau^{2}}{\varepsilon^{3}}, \qquad 0 \leq n \leq \frac{T}{\tau}.
\end{equation}
\end{theorem}

\begin{theorem}\label{1th1}
Under the assumptions in (A) and (B), there exist the constants $h_0>0,\,\tau_0>0$ independent of $\varepsilon$ and sufficiently small, such that for any $0<\varepsilon\le1,$ if $\,0<h \le h_0,\,0<\tau\le \tau_0$ and under the above stability condition \eqref{lefs1}, for the LFFD method \eqref{fuz1} with \eqref{zbcn} and \eqref{zf1si1}, we obtain the error estimate on the wave function as below
\begin{equation}\label{tc01}
\left\|\mathrm{\textbf{e}}^{n}\right\|_{l^{2}} \lesssim \frac{h^{2}}{\varepsilon}+\frac{\tau^{2}}{\varepsilon^{3}}, \qquad 0 \leq n \leq \frac{T}{\tau}.
\end{equation}
\end{theorem}

\begin{table}[thp]%[htp]%[H]%[htp]
 \caption{{\normalsize Spatial and temporal error analysis of the wave function $e_\varPhi^{h,\tau}(t=2)$ for the SIFD2 method}}
 \vspace{0.8mm}\label{sifd2ht}
 \centering
 \setlength{\tabcolsep}{5.5mm}%{5mm}
 %\resizebox{\textwidth}{15mm}
 {\begin{tabular}{cccccc}
 \hline
$e_\varPhi^h(t=2)$&$h_0=1/16$ &$h_0/2$  &$h_0/2^2$  &$h_0/2^3$ &$h_0/2^4$\\
 \hline\Xhline{0.8pt}
$\varepsilon_0=1$&\textbf{3.35E-1}&8.48E-2&2.12E-2&5.30E-3&1.33E-3\\
Order&$-$&1.98&2.00&2.00&2.00\\
$\varepsilon_0/4$&1.20&\textbf{3.22E-1}&8.11E-2&2.03E-2&5.07E-3\\
Order&$-$&\textbf{1.90}&1.99&2.00&2.00\\
$\varepsilon_0/4^{2}$&1.43&1.21&\textbf{3.22E-1}&8.09E-2&2.02E-2 \\
Order&$-$&0.24&\textbf{1.91}&1.99&2.00\\
$\varepsilon_0/4^{3}$&2.96&1.41&1.21&\textbf{3.23E-1}&8.09E-2\\
Order&$-$&1.07&0.22&\textbf{1.91}&1.99\\
\hline
$e_\varPhi^\tau(t=2)$&$\tau_0=1/40$ &$\tau_0/4$  &$\tau_0/4^2$  &$\tau_0/4^3$&$\tau_0/4^4$\\
\hline\Xhline{0.8pt}
$\varepsilon_0=1$&\textbf{1.21E-1}&7.68E-3&4.81E-4&3.03E-5&2.19E-6\\
Order&$-$&1.99&2.00&1.99&1.90\\
$\varepsilon_0/4^{2/3}$&1.48&\textbf{1.12E-1}&7.01E-3&4.39E-4&2.82E-5\\
Order&$-$&\textbf{1.87}&2.00&2.00&1.98\\
$\varepsilon_0/4^{4/3}$&3.15&1.52&\textbf{1.10E-1}&6.88E-3&4.32E-4\\
Order&$-$&0.53&\textbf{1.90}&2.00&2.00\\
$\varepsilon_0/4^{2}$&2.36&3.43&1.54&\textbf{1.09E-1}&6.85E-3\\
Order&$-$&-0.27&0.58&\textbf{1.91}&2.00\\
\hline
\end{tabular}}
\end{table}

By referring to the proof of Theorems \ref{4th4}, \ref{2th2}, the proof of Theorems \ref{3th3} and \ref{1th1} are similarly obtained. Actually, in the massless and nonrelativistic regime, when given an accuracy bound $\delta>0,$ the $\varepsilon$-resolution of the SIFD2 and LFFD methods is:
\begin{equation}\label{g1}
h=O(\sqrt{\delta \varepsilon})=O(\sqrt{\varepsilon}), \quad
\tau=O\left(\sqrt{\delta\varepsilon^{3}}\right)=O\left(\sqrt{\varepsilon^{3}}\right), \quad  0<\varepsilon \ll 1.
\end{equation}
Based on the Theorems \ref{4th4}-\ref{1th1}, the four finite difference methods analyzed here share the same spatial and temporal scalability for the Dirac equation in the massless and nonrelativistic regime.

\subsection{Numerical results}
In the following numerical simulation, the electromagnetic potential, initial condition, error functions and mesh sizes are same with \eqref{nte}-\eqref{abe} in Subsection \ref{cnfdnr}. Due to the stability conditions and accuracy requirement for the LFFD method, similar to the SIFD1 method, we take \eqref{stacop} in Table \ref{lffdht}. Tables \ref{sifd2ht}, \ref{lffdht} present the spatial errors $e_\varPhi^h(t=2)$ and the temporal errors $e_\varPhi^\tau(t=2)$ of the wave function for the SIFD2 \eqref{fuz3} and LFFD \eqref{fuz1} methods, respectively. Besides, Tables \ref{sifd2to}, \ref{sifd2cu} display the errors in spatial and temporal of the total density $e_\rho^{h/\tau}(t=2)$ and current density $e_\textbf{J}^{h/\tau}(t=2)$ by using the SIFD2 method.

\begin{table}[t!]%[htp]%[H!]%
 \caption{{\normalsize Spatial and temporal error analysis of the wave function $e_\varPhi^{h,\tau}(t=2)$ for the LFFD method}}
 \vspace{0.8mm} \label{lffdht}
 \centering
 \setlength{\tabcolsep}{5.2mm}%{5mm}
 %\resizebox{\textwidth}{15mm}
 {\begin{tabular}{cccccc}
 \hline
$e_\varPhi^h(t=2)$&$h_0=1/16$ &$h_0/2$  &$h_0/2^2$  &$h_0/2^3$ &$h_0/2^4$ \\
 \hline\Xhline{0.8pt}
$\varepsilon_0=1$&\textbf{3.35E-1}& 8.48E-2& 2.12E-2& 5.30E-3&1.33E-3\\
Order&$-$&1.98&2.00&2.00&2.00\\
$\varepsilon_0/4$&1.20&\textbf{3.22E-1}&8.11E-2&2.03E-2&5.07E-3\\
Order&$-$&\textbf{1.90}&1.99&2.00&2.00\\
$\varepsilon_0/4^{2}$&1.43&1.21&\textbf{3.22E-1}&8.08E-2&2.02E-2\\
Order&$-$&0.24&\textbf{1.91}&1.99&2.00\\
$\varepsilon_0/4^{3}$&2.96&1.41&1.21&\textbf{3.20E-1}&7.86E-2\\
Order&$-$&1.07&0.22&\textbf{1.92}&2.03\\
\hline
\multirow{2}*{$e_\Phi^\tau(t=2)$} &$\tau_0=1/40$ &$\tau_0/4$  &$\tau_0/4^2$  &$\tau_0/4^3$&$\tau_0/4^4$\\
&$h_0=1/16$ &$h_0/4\delta_1(\varepsilon)$ &$h_0/4^2\delta_2(\varepsilon)$ &$h_0/4^3\delta_3(\varepsilon)$&$h_0/4^4\delta_4(\varepsilon)$\\
\hline\Xhline{0.8pt}
$\varepsilon_0=1$&\textbf{2.80E-1}&1.73E-2&1.08E-3&6.75E-5&4.22E-6\\
Order&$-$&$2.01$&2.00&2.00&2.00\\
$\varepsilon_0/4^{2/3}$&Unstable&\textbf{1.48E-1}&9.22E-3&5.76E-4&3.60E-5\\
Order&$-$&$-$&2.00&2.00&2.00\\
$\varepsilon_0/4^{4/3}$&Unstable&1.13&\textbf{7.41E-2}&4.63E-3&2.89E-4\\
Order&$-$&$-$&\textbf{1.96}&2.00&2.00\\
$\varepsilon_0/4^{2}$&Unstable&1.55&4.24E-1&\textbf{2.64E-2}&1.65E-3\\
Order&$-$&$-$&0.93&\textbf{2.00}&2.00\\
\hline
\end{tabular}}
\end{table}

\begin{table}[t!]%[htp]%[H]
 \caption{{\normalsize Spatial and temporal error analysis of the total density $e_\rho^{h,\tau}(t=2)$ for the SIFD2 method}}
 \vspace{0.8mm}\label{sifd2to}
 \centering
 \setlength{\tabcolsep}{5.5mm}%{5mm}
 %\resizebox{\textwidth}{15mm}
 {\begin{tabular}{ccccccc}
 \hline\Xhline{0.8pt}
$e_\rho^h(t=2)$ &$h_0=1/32$ &$h_0/2$  &$h_0/2^2$  &$h_0/2^3$ &$h_0/2^4$\\
 \hline\Xhline{0.8pt}
$\varepsilon_0=1$&1.28E-1&3.02E-2&7.45E-3&1.86E-3&4.64E-4\\
Order&$-$&2.08&2.02&2.00&2.00\\
$\varepsilon_0/4$&5.83e-1&\textbf{1.52E-1}&3.85E-2&9.65E-3&2.41E-3\\
Order&$-$&\textbf{1.93}&1.99&2.00&2.00\\
$\varepsilon_0/4^{2}$&2.13&6.54E-1&\textbf{1.68E-1}&4.20E-2&1.05E-2\\
Order&$-$&1.70&\textbf{1.97}&2.00&2.00\\
$\varepsilon_0/4^{3}$&3.06&1.34&5.72E-1&\textbf{1.56E-1}&3.98E-2\\
Order&$-$&1.19&1.23&\textbf{1.87}&1.97\\
\hline
$e_\rho^\tau(t=2)$ &$\tau_0=1/40$ &$\tau_0/4$  &$\tau_0/4^2$  &$\tau_0/4^3$&$\tau_0/4^4$\\
\hline\Xhline{0.8pt}
$\varepsilon_0=1$&\textbf{1.91E-1}&1.13E-2&7.02E-4&4.42E-5&3.16E-6\\
Order&$-$&2.04&2.00&1.99&1.90\\
$\varepsilon_0/4^{2/3}$&1.71&\textbf{1.92E-1}&1.26E-2&7.93E-4&5.10E-5\\
Order&$-$&\textbf{1.58}&1.96&2.00&1.98\\
$\varepsilon_0/4^{4/3}$&3.65&3.24&\textbf{2.40E-1}&1.49E-2&9.34E-4\\
Order&$-$&0.09&\textbf{1.88}&2.01&2.00\\
$\varepsilon_0/4^{2}$&7.33&1.76&2.63&\textbf{2.25E-1}&1.42E-2\\
Order&$-$&1.03&-0.29&\textbf{1.77}&2.00\\
\hline
\end{tabular}}
\end{table}

\begin{table}[t!]%[htp]%[H]%
 \caption{{\normalsize Spatial and temporal error analysis of the current density $e_\textbf{J}^{h,\tau}(t=2)$ for the SIFD2 method}}
 \vspace{0.8mm}\label{sifd2cu}
 \centering
 \setlength{\tabcolsep}{5.5mm}%{5mm}
 %\resizebox{\textwidth}{15mm}
 {\begin{tabular}{ccccccc}
 \hline\Xhline{0.8pt}
$e_\textbf{J}^h(t=2)$ &$h_0=1/32$ &$h_0/2$  &$h_0/2^2$  &$h_0/2^3$ &$h_0/2^4$\\
 \hline\Xhline{0.8pt}
$\varepsilon_0=1$&4.67E-1&1.28E-1&3.21E-2&8.02E-3&2.01E-3\\
Order&$-$&1.87&2.00&2.00&2.00\\
$\varepsilon_0/4$&2.05&\textbf{5.77E-1}&1.38E-1&3.40E-2&8.47E-3\\
Order&$-$&\textbf{1.83}&2.06&2.02&2.01\\
$\varepsilon_0/4^{2}$&1.42&1.30&\textbf{3.49E-1}&8.67E-2&2.16E-2\\
Order&$-$&0.13&\textbf{1.89}&2.01&2.01\\
$\varepsilon_0/4^{3}$&3.08&2.38&2.43&\textbf{6.52E-1}&1.56E-1\\
Order&$-$&0.37&-0.03&\textbf{1.90}&2.06\\
\hline
$e_\textbf{J}^\tau(t=2)$ &$\tau_0=1/40$ &$\tau_0/4$  &$\tau_0/4^2$  &$\tau_0/4^3$&$\tau_0/4^4$\\
\hline\Xhline{0.8pt}
$\varepsilon_0=1$&\textbf{1.59E-1}&1.00E-2&6.25E-4&3.95E-5&2.86E-6\\
Order&$-$&2.00&2.00&1.99&1.89\\
$\varepsilon_0/4^{2/3}$&1.87&\textbf{1.23E-1}&7.40E-3&4.62E-4&2.96E-5\\
Order&$-$&\textbf{1.97}&2.03&2.00&1.98\\
$\varepsilon_0/4^{4/3}$&1.70&9.65E-1&\textbf{8.95E-2}&5.74E-3&3.61E-4\\
Order&$-$&0.41&\textbf{1.72}&1.98&2.00\\
$\varepsilon_0/4^{2}$&2.98&6.78E-1&1.38&\textbf{9.10E-2}&5.64E-3\\
Order&$-$&1.07&-0.51&\textbf{1.96}&2.01\\
\hline
\end{tabular}}
\end{table}

From Tables \ref{sifd2ht}-\ref{sifd2cu}, we can directly observe that the LFFD and SIFD1 methods have second order convergence both in time and space, and the $\varepsilon$-resolution $h=O(\varepsilon^{1/2})$ and $\tau=O(\varepsilon^{3/2})$ for the wave function and two densities are consistent with Theorems \ref{3th3}, \ref{1th1} and Corollaries \ref{1co1}, \ref{1co2}. The error estimates are verified by the upper triangles above the diagonal line labelled with bold  type in the top and bottom half of each table. Analogously, the two densities for the CNFD, SIFD1 and LFFD have similar results, which are omitted here for brief.

According to the numerical results presented above, in the massless and nonrelativistic regime, we successfully verify the error estimates for the wave function, total and current densities of the Dirac equation in Theorems \ref{4th4}-\ref{1th1} and Corollaries \ref{1co1}, \ref{1co2} by using the finite difference methods. Moreover, we could obtain our error estimates of Theorems and Corollaries are sharp.

%the display of other results is omitted here for brief.
\section{Conclusion}
In this paper, we use four types of finite difference methods numerically to study the Dirac equation in the massless and nonrelativistic regime. The four finite difference methods, including the energy nonconservative/conservative and fully explicit/two semi-implicit/implicit numeric schemes, all have second order convergence in both space and time. In the massless and nonrelativistic regime, the corresponding stability conditions and error estimates of these discrete schemes are rigorously analyzed respectively. The error estimates suggest that the wave function, total density and current density of these four finite difference methods share the same $\varepsilon$-scalability as $h=O(\varepsilon^{1/2})$ and $\tau=O(\varepsilon^{3/2})$. Extensive numerical results are exhibited to verify our error estimates. From the above analysis and numerical examples, it is clear that the computational cost of the CNFD method is the most expensive, while the LFFD method is the most efficient, but it has the most strict stability condition.

\section*{Acknowledgments}
This work was supported by the National Natural Science Foundation of China Grant U1930402 (Y. Ma), and the Ministry of Education of Singapore grant R-146-000-290-114 (J. Yin).
Part of this work was done when the authors visited the Institute for Mathematical Sciences at the National University of Singapore in 2019.
%The authors would like to thank Professor Weizhu Bao for his valuable suggestions.

% Here are two sample references: \cite{Feynman1963118,Dirac1953888}.

%\bibliographystyle{siamplain}
%\bibliographystyle{unsrt}
%\bibliographystyle{plain}
%\bibliography{./ref}{}
%%%\bibliographystyle{abbrv}
%%%\bibliography{ref}
%thesisref;plain;unsrt;amsplain;siam;unsrt:abbrv
%\bibliography{mybibfile}

\end{document}